\documentclass[12pt]{article}
\usepackage[margin=1in]{geometry}
\usepackage[utf8]{inputenc}

\usepackage[english]{babel}

\usepackage[square,numbers,sort&compress]{natbib}
\usepackage{mathtools,amssymb,amsthm}
\usepackage{graphicx,color}
\usepackage{subcaption}
\usepackage{pifont}
\usepackage{booktabs}

\usepackage{float}
\usepackage[boxed]{algorithm2e}

\usepackage[toc,page]{appendix}
\usepackage{hyperref}
\usepackage[capitalize]{cleveref}
\usepackage{bm}
\usepackage{enumitem}
\usepackage{mathtools,amsmath,amsthm, amssymb}
\usepackage{xifthen}
\usepackage{dsfont}

\newcommand{\cG}{\mathcal{G}}

\newcommand{\cN}{\mathcal{N}}

\newcommand{\cP}{\mathcal{P}}

\newcommand{\RR}{\mathbb{R}}

\newcommand{\R}{\RR}
\newcommand{\E}{\mathbb{E}}

\newcommand*{\kl}[3][]{%
\ifthenelse{\isempty{#1}}{\operatorname{D}(#2\,\|\,#3)}%
{\operatorname{D}(#2\,\|\,#3\mid#1)}%
}

\newcommand*{\triplenorm}[1]{{\left\vert\kern-0.25ex\left\vert\kern-0.25ex\left\vert #1
    \right\vert\kern-0.25ex\right\vert\kern-0.25ex\right\vert}}

\newcommand*{\defeq}{\coloneqq}
\newcommand*{\rd}{\mathrm{d}}
\newcommand*{\dd}{\, \rd}

\newcommand{\eps}{\varepsilon}

\renewcommand{\phi}{\varphi}

\newcommand{\mmid}{\,\Vert\,}
\newcommand{\deq}{\coloneqq}
\usepackage{comment,url,graphicx,subcaption,relsize}
\usepackage{amssymb,amsfonts,amsmath,amsthm,amscd,dsfont,mathrsfs,mathtools,multirow,microtype,nicefrac,pifont}
\usepackage{float,psfrag,epsfig,color,url,hyperref}
\usepackage{upgreek}
\usepackage[dvipsnames]{xcolor}
\usepackage{epstopdf,bbm,mathtools,enumitem}
\usepackage[toc,page]{appendix}
\usepackage[mathscr]{euscript}

\def\balign#1\ealign{\begin{align}#1\end{align}}
\def\baligns#1\ealigns{\begin{align*}#1\end{align*}}
\def\balignat#1\ealign{\begin{alignat}#1\end{alignat}}
\def\balignats#1\ealigns{\begin{alignat*}#1\end{alignat*}}
\def\bitemize#1\eitemize{\begin{itemize}#1\end{itemize}}
\def\benumerate#1\eenumerate{\begin{enumerate}#1\end{enumerate}}

\newenvironment{talign*}
 {\csname align*\endcsname}
 {\endalign}
\newenvironment{talign}
 {\csname align\endcsname}
 {\endalign}

\def\balignst#1\ealignst{\begin{talign*}#1\end{talign*}}
\def\balignt#1\ealignt{\begin{talign}#1\end{talign}}

\let\originalleft\left
\let\originalright\right
\renewcommand{\left}{\mathopen{}\mathclose\bgroup\originalleft}
\renewcommand{\right}{\aftergroup\egroup\originalright}

\def\tinycitep*#1{{\tiny\citep*{#1}}}
\def\tinycitealt*#1{{\tiny\citealt*{#1}}}
\def\tinycite*#1{{\tiny\cite*{#1}}}
\def\smallcitep*#1{{\scriptsize\citep*{#1}}}
\def\smallcitealt*#1{{\scriptsize\citealt*{#1}}}
\def\smallcite*#1{{\scriptsize\cite*{#1}}}

\def\<{\left\langle} 
\def\>{\right\rangle}

\def\defeq{\coloneqq} 
\DeclareSymbolFont{rsfs}{U}{rsfs}{m}{n}
\DeclareSymbolFontAlphabet{\mathscrsfs}{rsfs}
\ifdefined\nonewproofenvironments\else
\ifdefined\ispres\else
\renewenvironment{proof}{\noindent\textbf{Proof.}\hspace*{.3em}}{\qed \vspace{.1in}}
\newenvironment{proof-sketch}{\noindent\textbf{Proof Sketch}
  \hspace*{1em}}{\qed\bigskip\\}
\newenvironment{proof-idea}{\noindent\textbf{Proof Idea}
  \hspace*{1em}}{\qed\bigskip\\}
\newenvironment{proof-of-lemma}[1][{}]{\noindent\textbf{Proof of Lemma {#1}}
  \hspace*{1em}}{\qed\\}
  \newenvironment{proof-of-proposition}[1][{}]{\noindent\textbf{Proof of Proposition {#1}}
  \hspace*{1em}}{\qed\\}
\newenvironment{proof-of-theorem}[1][{}]{\noindent\textbf{Proof of Theorem {#1}}
  \hspace*{1em}}{\qed\\}
\newenvironment{proof-attempt}{\noindent\textbf{Proof Attempt}
  \hspace*{1em}}{\qed\bigskip\\}

\fi

\fi
\makeatletter
\@addtoreset{equation}{section}
\makeatother

\hypersetup{
  colorlinks,
  linkcolor={red!50!black},
  citecolor={blue!50!black},
  urlcolor={blue!80!black}
}

\newcommand{\infFI}{\operatorname{FI}_\infty}
\global\long\def\esssup{{\rm{ess\,sup}}}%

\theoremstyle{plain}
\newtheorem{theorem}{Theorem}[section]
\newtheorem{proposition}[theorem]{Proposition}

\newtheorem{lemma}[theorem]{Lemma}
\newtheorem{corollary}[theorem]{Corollary}
\theoremstyle{definition}

\theoremstyle{remark}
\newtheorem{remark}[theorem]{Remark}

\title{Stability of the Kim--Milman flow map}

\author{
Sinho Chewi\thanks{Department of Statistics and Data Science, Yale University. \tt{sinho.chewi@yale.edu}} \and 
Aram-Alexandre Pooladian\thanks{Institute for Foundations of Data Science, Yale University. \tt aram-alexandre.pooladian@yale.edu}\and
Matthew S. Zhang\thanks{Department of Computer Science, University of Toronto. \tt matthew.zhang@mail.utoronto.ca} 
}

\begin{document}
\maketitle
\abstract{In this short note, we characterize stability of the Kim--Milman flow map---also known as the probability flow ODE---with respect to variations in the target measure {in relative Fisher information}. 
}

\section{Introduction}\label{sec:intro}
In general, there are infinitely many maps which transport a fixed source distribution $\rho$ to a target distribution $\mu$, {where both measures are} in $\cP(\R^d)$. Let $T_\bullet$ denote a method of generating such transport maps; thus, $T_{\bullet}^{\rho\to\mu}$ is a transport map from $\rho$ to $\mu$ for every $\mu \in \cP(\R^d)$, meaning that for $X\sim\rho$, $T_\bullet^{\rho\to\mu}(X)\sim\mu$. A property of fundamental interest for such a method is its stability with respect to variations in the target measure: how much do the transport maps vary if the target measures vary? That is, for another target measure $\nu$ and under a broad class of assumptions, we want to understand inequalities of the form \looseness-1
\begin{align}\label{eq:stab_intro}
\|T^{\rho\to\mu}_\bullet - T_\bullet^{\rho\to\nu}\|^2_{L^2(\rho)} \lesssim \mathrm{D}(\mu,\nu)\,,
\end{align}
where $\rm{D}(\mu,\nu)$ is some dissimilarity metric between the two target measures.

To the best of our knowledge, the study of inequalities of the form \eqref{eq:stab_intro} has been limited to \emph{optimal} transport maps \cite{gigli2011holder, kitagawa2025stability,
letrouit2024gluing,
balakrishnan2025stability}, denoted $T_{\rm OT}^{\rho\to\mu}$, or \emph{entropic} transport maps \cite{carlier2024displacement,
divol2025tight}, denoted $T_{\rm EOT}^{\rho\to\mu}$. In these instances, the natural dissimilarity metric becomes the (squared) $2$-Wasserstein distance between $\mu$ and $\nu$, and the underlying constant depends on properties of the source $\rho$, and either the class of target measures $\mu$ or a priori assumptions on the (entropic) optimal transport map.
Existing bounds are of the form\looseness-1
\begin{align*}
\|T^{\rho\to\mu}_{\rm OT} - T_{\rm OT}^{\rho\to\nu}\|_{L^2(\rho)} \leq C W_2^{{\beta}}(\mu,\nu)\,, \quad \text{ or } \quad \|T^{\rho\to\mu}_{\rm EOT} - T_{\rm EOT}^{\rho\to\nu}\|_{L^2(\rho)} \leq C_\eps W_2(\mu,\nu)\,,
\end{align*}
where {$\beta \in (0,1]$ and}, for the entropic maps, $\eps > 0$ is the regularization parameter and  $C_\eps \nearrow +\infty$ as $\eps \searrow 0$; see Section~\ref{sec:related} for more information. 

In this note, we study stability properties of a different transport map called the Kim--Milman (reverse) heat flow map \cite{kim2012generalization}, which has recently gained popularity in {mathematics \cite{neeman2022lipschitz,klartag2023spectral,MikShe23HeatFlow,lopez2025bakry,BriPed25HeatFlow,fathi2024transportation,conforti2025coupling,serres2026contractive,ge2025generalization}} as well as in the machine learning literature under the moniker ``probability flow ODE'' \cite{Song+21SDE, Kim+24Consistency} due to its efficacy in generative modeling tasks. Unlike optimal or entropic transport maps, we stress that this transport map is defined dynamically. As a brief description, let $\gamma$ denote the $d$-dimensional standard Gaussian distribution, and let $\mu$ be another $d$-dimensional probability distribution; we want to study the stability of the following (reverse) ODE system \looseness-1
\begin{align*}
    \dot{X}_t &= X_t + \nabla \log \mu Q_{T-t}(X_t) = \nabla \log Q_{T-t}\Bigl[\frac{\mu}{\gamma}\Bigr](X_t)\,,
\end{align*}
where $0 \ll T < \infty$, $Q_{s}$ is the Ornstein--Uhlenbeck semigroup at time $s > 0$, and we initialize $X_0 \sim \mu Q_T \approx \gamma$. Let $T_{\rm KM}^\mu$ be the flow map of the ODE system with $T = \infty$ (rigorously, the limit of the flow maps up to time $T$, taking $T \nearrow \infty$), which transports $(T_{\rm KM}^\mu)_\sharp\gamma = \mu$. By imposing regularity assumptions on $\mu$, we will prove that 
\begin{align}\label{eq:result_intro}
\|T^{\mu}_{\rm KM} - T_{\rm KM}^{\nu}\|_{L^2(\gamma)} \lesssim \sqrt{\rm{FI}(\nu\mmid \mu)}\,,
\end{align}
where $\rm{FI}(\nu\mmid\mu) \deq \|\nabla\log(\nu/\mu) \|^2_{L^2(\nu)}$, where the underlying constant is explicit{; see Theorem~\ref{thm:l2_result}}. We then show how our analysis can be used to prove bounds of the form
\begin{align}\label{eq:result_intro_2}
\|T^{\mu}_{\rm KM} - T_{\rm KM}^{\nu}\|_{L^\infty(\gamma)}\lesssim \sqrt{\infFI(\nu\mmid \mu)}\,,
\end{align}
where $\infFI(\nu\mmid\mu) \defeq \esssup_\nu \|\nabla \log(\nu/\mu)\|^2${; see Theorem~\ref{thm:inffi}}. We consider as applications the case where $\mu$ is a perturbation of a strongly log-concave measure (leveraging recent results by \cite{BriPed25HeatFlow, wibisono2025mixing}) and when $\mu$ has an asymptotically positive convexity profile (as introduced by \cite{conforti2024weak}). We stress that, at present, these cases are not covered by stability results for (entropic) optimal transport maps. 

{For future work, it would be natural to extend our results to measures which live on Riemannian manifolds \cite{fathi2024transportation} or to other flow maps (e.g., the flow map induced by general stochastic interpolants \cite{Lip+23FlowMatching, AlbBofVan25StochInterpolants} or the Schr\"odinger bridge flow map  \cite{Leo14Schrodinger, shi2023diffusion, PooNil25SB}).}

\subsubsection*{Acknowledgments}
We thank Zhou Fan and Daniel Lacker for helpful discussions{, as well as Katharina Eichinger for pointing out a technical flaw in our original proof, and the referees for providing valuable feedback}. AAP thanks the Foundations of Data Science at Yale University for support. MSZ is funded by an NSERC CGS-D award.

\subsubsection*{Notation}

We write $\cP(\R^d)$ for the space of probability measures over $\R^d$.
Given a probability measure $\mu$ which admits a Lebesgue density, we abuse notation and write $\mu$ for the density as well as for the measure.
For a symmetric matrix $\Sigma$, we write $\|x\|_\Sigma \deq \sqrt{\langle x, \Sigma\,x\rangle}$.

\section{Background}\label{sec:background}
\subsection{{Kim--Milman flow map}}\label{sec:pfode_bg}
We now {sketch the derivation of the reverse heat flow due to Kim and Milman \cite{kim2012generalization}, also known as the probability flow ODE}.

For ${X}^{\rightarrow}_0 \sim \mu \in \cP(\R^d)$, recall the forward Ornstein--Uhlenbeck process
\begin{align*}
    {\rm d}{X}^{\rightarrow}_t = -{X}^{\rightarrow}_t {\rm d}t + \sqrt{2}\,{\rm d}B_t\,,
\end{align*}
where $(B_t)_{t\ge 0}$ is standard Brownian motion. Note that as $t\to\infty$, $q_t \deq {\rm Law}(X_t^{\rightarrow})\to\gamma = \cN(0,I)$. Now, run the stochastic differential equation (SDE) for time $0 \ll T < +\infty$. Then, the reverse SDE system is given by
\begin{align}\label{eq:rev_OU}
    {\rm d}{X}^{\leftarrow}_t = \bigl({X}^{\leftarrow}_t + 2\, \nabla \log q_{T-t}({X}^{\leftarrow}_t)\bigr)\,{\rm d}t + \sqrt{2}\,{\rm d}B_t\,, \qquad X_0^\leftarrow \sim q_T\,,
\end{align}
where ${\rm Law}(X^\leftarrow_s) = q_{T-s}$.
(Precise conditions for the well-posedness of the time reversal can be found in~\cite{Cat+23TimeReversal}.)
Note that the Brownian motion is also reversed. The corresponding Fokker--Planck equation for the reverse SDE is then
\begin{align*}
\partial_t q_{T-t} + \nabla\cdot(q_{T-t}\, ({\rm{id}} + 2\, \nabla \log q_{T-t})) = \Delta q_{T-t} = \nabla \cdot(q_{T-t} \nabla \log q_{T-t})\,.
\end{align*}
We can incorporate the diffusion term above into the drift, resulting in the continuity equation
\begin{align*}
\partial_t q_{T-t} + \nabla\cdot(q_{T-t}\, ({\rm{id}} +  \nabla \log q_{T-t})) = 0\,,
\end{align*}
which describes the evolution of the marginal law of the ODE system
\begin{align}\label{eq:kimmilman}
    \dot X^{\leftarrow}_t = {X}^{\leftarrow}_t + \nabla \log q_{T-t}({X}^{\leftarrow}_t)\,, \qquad X_0^\leftarrow \sim q_T\,.
\end{align}
Note that while~\eqref{eq:rev_OU} and~\eqref{eq:kimmilman} differ as stochastic processes, by construction they have the same marginal laws $(q_{T-t})_{t\in [0,T]}$.

Let $T_{\rm KM}^{\mu,T}$ be the flow map corresponding to the ODE system~\eqref{eq:kimmilman}; thus, $(T_{\rm KM}^{\mu, T})_\sharp q_T = \mu$.
Finally, we let $T_{\rm KM}^\mu \deq \lim_{T\to\infty} T_{\rm KM}^{\mu,T}$ denote the Kim--Milman map.

\subsection{Related work}\label{sec:related}
\paragraph{Stability of optimal transport maps.} Stability of optimal transport maps was first studied by Gigli in~\cite{gigli2011holder}. His main result states that if one of the transport maps, say $T^{\rho\to\mu}_{\rm OT}$, is $L$-Lipschitz and the support of $\mu$ is compact (say in a ball with radius $R$), then
\begin{align*}
\|T^{\rho\to\mu}_{\rm OT} - T_{\rm OT}^{\rho\to\nu}\|_{L^2(\rho)} \lesssim (LR)^{1/2}\, W_2^{1/2}(\mu,\nu)\,.
\end{align*}
{Using different arguments, Delalande and Merigot \cite{delalande2023quantitative} showed that, so long as $\mu,\nu$ have compact support and $\rho$ has density bounded above and below on a compact convex subset of $\R^d$, then 
\begin{align*}
\|T^{\rho\to\mu}_{\rm OT} - T_{\rm OT}^{\rho\to\nu}\|_{L^2(\rho)} \lesssim W_2^{1/6}(\mu,\nu)\,.
\end{align*}
Letrouit and M\'erigot in~\cite{letrouit2024gluing} pushed these arguments to accommodate more general source measures, and \cite{kitagawa2025stability} proposed an extension to Riemannian manifolds.} 
On the other hand, it is not possible to establish H\"older stability in general with an exponent better than $1/2$, due to the counterexample in~\cite{gigli2011holder}.

Going a step further, Manole, Balakrishnan, Niles-Weed, and Wasserman~\cite{manole2024plugin} show that if one of the optimal transport maps is bi-Lipschitz, i.e., if $0 \prec \ell I \preceq D T^{\rho\to\mu}_{\rm OT} \preceq L I$, then we have the stronger bound
\begin{align*}
\|T^{\rho\to\mu}_{\rm OT} - T_{\rm OT}^{\rho\to\nu}\|_{L^2(\rho)} \leq \Bigl(\frac{L}{\ell}\Bigr)^{1/2}\, W_2(\mu,\nu)\,.
\end{align*}
For more results on this topic, see the recent monograph by Letrouit \cite{letrouit2025quantitative}. 

\paragraph{Stability of entropic transport maps.}
Entropic transport maps were recently introduced in \cite{pooladian2021entropic} to estimate optimal transport maps from samples; see also \cite{rigollet2025sample,
pooladian2023minimax}. The stability of entropic transport maps is more recent, first investigated by Carlier, Chizat, and Laborde \cite{carlier2024displacement}. Specializing their results, they prove that if all measures $\rho,\mu,\nu$ have compact support, then  
\begin{align*}
    \|T^{\rho\to\mu}_{\rm EOT} - T_{\rm EOT}^{\rho\to\nu}\|_{L^2(\rho)} \lesssim \exp(c/\eps)\, W_2(\mu,\nu)\,,
\end{align*}
where $c > 0$ is a constant depending on the diameter of the supports. Crucially, as $\eps \searrow 0$, the results for optimal transport maps are not recovered in any regime. Recently, a tight stability bound for entropic transport maps was proven by Divol, Niles-Weed, and {the second author}~\cite{divol2025tight}, showing that
\begin{align*}
    \|T^{\rho\to\mu}_{\rm EOT} - T_{\rm EOT}^{\rho\to\nu}\|_{L^2(\rho)} \lesssim \frac{R^2}{\eps}\,W_2(\mu,\nu)\,,
\end{align*}
where the sole requirement is that all three measures have compact support (say in a ball of radius $R > 0$). However, the results of \cite{divol2025tight} are more general. For example, if one of the entropic transport maps is bi-Lipschitz (which can be ensured in certain situations \cite{chewi2023entropic}), then they are able to recover the results of \cite{manole2024plugin} by taking $\eps \searrow 0$.\looseness-1

\section{Main results}\label{sec:main_results}
We are interested in understanding the dynamics of the system \eqref{eq:kimmilman}. To this end, let $(X_t)_{t\in[0,T]}$ (resp.\ $(Y_t)_{t\in[0,T]}$) be the reverse dynamics from ${q_T^\mu}$ to $\mu$ (resp.\ ${q_T^\nu}$ to $\nu$) for $0\ll T < +\infty$ with $q^\mu_{{T-t}} \defeq {\rm{Law}}(X_t)$ (resp.\ $q^\nu_{{T-t}} \defeq {\rm{Law}}(Y_t)$). Thus, $T_{\rm KM}^\mu(x)$ is the terminal point $X_T$ of the ODE~\eqref{eq:kimmilman} with $T\to\infty$ and initialized at $x$, and similarly for $T_{\rm KM}^\nu(x)$.
For a rigorous justification of this procedure, see, e.g., \cite{kim2012generalization, MikShe23HeatFlow}. Our calculations will require that for all $s \geq 0$
\begin{align}\label{eq:assump}\tag{$\Theta$}
\nabla^2 \log Q_{s}\Bigl[\frac{\mu}{\gamma}\Bigr] (\cdot) = I + \nabla^2 \log q_s^\mu(\cdot) \preceq {\theta}_s I
\end{align}
{for some constants $(\theta_s)_{s \ge 0}$.} Our main results will be of the form
\begin{align*}
    \|T_{\rm KM}^\mu - T_{\rm KM}^\nu \|_{L^2(\gamma)} \lesssim \sqrt{{\rm FI}(\nu\mmid\mu)}\,,
\end{align*}
{where the omitted constants will explicitly depend on assumptions on $\mu$.}
\subsection{Main computation}\label{sec:main_computation}
{Since we only place assumptions on $\mu$, even the existence of $T_{\rm KM}^\nu$ is non-trivial. To avoid technical complications, we assume that $T_{\rm KM}^\mu$, $T_{\rm KM}^\nu$ are well-defined in the sense that $T_{\rm KM}^{\mu,T} \to T_{\rm KM}^\mu$ uniformly on compact sets and similarly for $T_{\rm KM}^\nu$.\footnote{{This is stronger than the convergence established in~\cite{MikShe23HeatFlow}.}}} {For more complete discussions, see \cite{kim2012generalization,MikShe23HeatFlow}.}

In this section, we prove our main theorem.

\begin{theorem}\label{thm:l2_result}
Let $\mu$ be such that \eqref{eq:assump} is satisfied for some constants $(\theta_s)_{s \geq 0}$.
{Assume
\begin{align*}
    \Lambda_\infty \defeq \lim_{T\to\infty} \Lambda_T < \infty\,, \qquad\text{where}~\Lambda_T \defeq \int_0^T \exp\Bigl(\int_0^t (3\theta_s-1)\dd s\Bigr)\dd t\,.
\end{align*}
}
Then, for any $\nu$ such that $\rm{FI}(\nu\mmid\mu) < \infty$,
\begin{align}\label{eq:main_result_equation}
    \|T_{\rm KM}^\mu - T_{\rm KM}^\nu\|_{L^2(\gamma)} \leq \Lambda_\infty \sqrt{{\rm FI}(\nu\mmid \mu)}\,.
\end{align}
\end{theorem}
To prove our main theorem, we require the following lemma due to Wibisono \cite{wibisono2025mixing}.
\begin{lemma}[{\citep[Theorem 4(i)]{wibisono2025mixing}}]\label{lem:andre}
For any $t > 0$,
\begin{align*}
    \frac{{\rm d}}{{\rm d}t}\, {\rm FI}(q_{t}^\nu\mmid q_t^\mu)\leq -2\, \E_{q_t^\nu}\Bigl\|\nabla \log \frac{q^\nu_{t}}{q_t^\mu}\Bigr\|^2_{(-2\nabla^2 \log q_t^\mu {-} I)}\,.
\end{align*}
\end{lemma}

\begin{proof}[Proof of Theorem~\ref{thm:l2_result}]
To start, we compute
\begin{align*}
    \partial_t \E\|X_t - Y_t\|^2 &= 2\, \E \langle X_t - Y_t, \dot{X}_t - \dot{Y}_t\rangle \\
    &= 2\,\E\Bigl\langle X_t - Y_t , \nabla \log Q_{T-t}\Bigl[\frac{\mu}{\gamma}\Bigr](X_t) - \nabla \log Q_{T-t}\Bigl[\frac{\nu}{\gamma}\Bigr](Y_t)\Bigr\rangle  \\
    &\leq  2\theta_{T-t}\,\E\|X_t - Y_t\|^2 \\
    &\qquad{} + 2\,\E \Bigl\langle X_t - Y_t, \nabla \log Q_{T-t}\Bigl[\frac{\mu}{\gamma}\Bigr](Y_t) - \nabla \log Q_{T-t}\Bigl[\frac{\nu}{\gamma}\Bigr](Y_t) \Bigr\rangle\,,
\end{align*}
where we used \eqref{eq:assump} in the last inequality. Using Cauchy--Schwarz, we obtain 
\begin{align}\label{eq:inequ}
    \partial_t \E\|X_t - Y_t\|^2 &\leq 2{\theta_{T-t}}\,\E\|X_t - Y_t\|^2 + 2\,(\E\|X_t - Y_t\|^2)^{1/2}\,{\rm{FI}}(q^\nu_{T-t}\mmid q^\mu_{T-t})^{1/2}\,,
\end{align}
where the relative Fisher information makes an appearance. Now, as $\nabla^2 \log q_t^\mu = \nabla^2 \log Q_t\bigl[\frac{\mu}{\gamma}\bigr] - I \preceq {(\theta_t-1)}\,I$, Lemma~\ref{lem:andre} simplifies to 
\begin{align*}
    \frac{{\rm d}}{{\rm d}t} {\rm FI}(q_{t}^\nu\mmid q_t^\mu)\leq -2\,(-2\,(\theta_t - {1})-1)\, {\rm FI}(q_t^\nu\mmid q_t^\mu) = -2\,(-2\theta_t + {1})\,{\rm FI}(q_t^\nu\mmid q_t^\mu)\,,
\end{align*}
via \eqref{eq:assump} and thus, by Gr\"onwall's inequality
\begin{align*}
    {\rm FI}(q_{t}^\nu\mmid q_t^\mu) &\leq \exp\Bigl(- 2\int_0^t (-2\theta_u+ {1}) \dd u \Bigr)\,{\rm FI}(q_0^\nu\mmid q_0^\mu) = \exp\Bigl(- 2\int_0^t (-2\theta_u+ {1}) \dd u \Bigr)\,{\rm FI}(\nu\mmid \mu)\,.
\end{align*}
Since the above holds for any time $t > 0$, we choose $T-t$, and thus
\begin{align*}
    \sqrt{{\rm FI}(q_{T-t}^\nu\mmid q_{T-t}^\mu)} \leq e^{-(T-t)}\exp\Bigl( \int_0^{T-t} 2\theta_u \dd u \Bigr)\,{\rm FI}(\nu\mmid \mu)^{1/2} \eqqcolon {c_{T-t}}\sqrt{{\rm FI}(\nu\mmid \mu)}\,. 
\end{align*}
Using the fact that
\begin{align*}
    \partial_t (\E\|X_t-Y_t\|^2)^{1/2} = \frac{1}{2}\,\frac{\partial_t \E\|X_t-Y_t\|^2}{(\E\|X_t-Y_t\|^2)^{1/2}}\,,
\end{align*}
we obtain
\begin{align*}
\partial_t (\E\|X_t - Y_t\|^2)^{1/2} \leq {\theta_{T-t}}\,(\E\|X_t - Y_t\|^2)^{1/2} + c_{T-t}\sqrt{{\rm FI}(\nu\mmid\mu)}\,.
\end{align*}
Applying Gr\"onwall's inequality again gives
\begin{align*}
(\E\|X_t-Y_t\|^2)^{1/2}
&\leq {\exp\Bigl(\int_0^t \theta_{T-s}\dd s\Bigr)\, (\E\|X_0-Y_0\|^2)^{1/2}} \\
&\qquad + \sqrt{{\rm FI}(\nu\mmid\mu)}\int_0^t c_{T-s}\exp\Bigl(\int_s^t {\theta_{T-r}} \dd r\Bigr) \dd s\,.
\end{align*}
{Let $Z \sim \gamma$ and $X_0 \defeq e^{-T} \bar X + \sqrt{1-e^{-2T}}\, Z$, $Y_0 \defeq e^{-T} \bar Y + \sqrt{1-e^{-2T}}\,Z$, where $\bar X \sim \mu$, $\bar Y\sim \nu$ are coupled so that $\|\bar X - \bar Y\|_{L^2} = W_2(\mu,\nu)$.
This produces a valid coupling $X_0\sim q_T^\mu$, $Y_0 \sim q_T^\nu$.
Moreover, $X_0 \to Z$ almost surely as $T\to\infty$, and similarly $Y_0 \to Z$.
By the assumption that the Kim--Milman maps are well-defined, it holds that $T_{\rm KM}^{\mu,T}(X_0) \to T_{\rm KM}^\mu(Z)$ and $T_{\rm KM}^{\nu,T}(Y_0) \to T_{\rm KM}^\nu(Z)$.
By Fatou's lemma and the above bound,
\begin{align*}
    \|T_{\rm KM}^\mu(Z) - T_{\rm KM}^\nu(Z)\|_{L^2}
    &\le \liminf_{T\to\infty} \|T_{\rm KM}^{\mu,T}(X_0) - T_{\rm KM}^{\nu,T}(Y_0)\|_{L^2} \\
    &\le \liminf_{T\to\infty}\Bigl[\exp\Bigl(\int_0^T (\theta_t -1)\dd t\Bigr)\,W_2(\mu,\nu) \\
    &\qquad{} + \sqrt{{\rm FI}(\nu\mmid\mu)}\int_0^T c_{T-s}\exp\Bigl(\int_s^T {\theta_{T-r}} \dd r\Bigr) \dd s\Bigr]\,.
\end{align*}
Now, we argue that the first term vanishes.
We can safely assume $\Lambda_\infty < \infty$, or else the statement we wish to prove is trivial.
Then, $\Lambda_\infty \ge \int_0^\infty \exp(3\int_0^T (\theta_u-1)\dd u)\dd T$ which implies that $\exp(\int_0^T (\theta_u-1)\dd u) \to 0$ as $T\to\infty$ along a subsequence, as desired.
}

{Now taking $t=T$ and writing out $c_{T-s}$, our proof concludes by taking a limit}
\begin{align*}
   {\frac{\|T_{\rm KM}^\mu - T_{\rm KM}^\nu\|_{L^2(\gamma)}}{\sqrt{{\rm FI}(\nu\mmid\mu)}}} &\leq {\limsup_{T\to\infty}} \int_0^T e^{-(T-s)}\exp\Bigl( \int_0^{T-s} 2\theta_u \dd u \Bigr) \exp\Bigl(\int_s^T {\theta_{T-r}} \dd r\Bigr) \dd s \\
   &= {\limsup_{T\to\infty}} \int_0^T e^{-(T-s)}\exp\Bigl(3 \int_0^{T-s} \theta_u \dd u \Bigr) \dd s \\
   &= {\limsup_{T\to\infty}\int_0^T \exp\Bigl( \int_0^s (3\theta_u-1) \dd u \Bigr) \dd s\,,}
\end{align*}
where we apply a change of variables.
\end{proof}
\subsubsection{Strong log-concavity with log-Lipschitz perturbations}\label{sec:slc_pert}
We {first} suppose our main target measure is of the form $\mu\propto\exp(-V+H)$ where $V$ is $\alpha$-strongly convex and $H$ is a (smooth) $L$-Lipschitz perturbation. In this setting, a recent result of \cite{BriPed25HeatFlow} showed that
\begin{align}\label{eq:bartheta_bp24}
    \theta_{u} = \frac{1-\alpha}{\alpha\,(e^{2u} - 1) + 1} + \frac{e^{2u}L^2}{(\alpha\,(e^{2u}-1) + 1)^2} + \frac{2L e^{2u}}{(\alpha\,(e^{2u}-1) + 1)^{3/2}\sqrt{e^{2u}-1}}\,.
\end{align}
Carrying out the algebra, the complete stability bound in this setting is given below. 
\begin{corollary}[Perturbation of strongly log-concave]\label{cor:pert}
Suppose $\mu \propto\exp(-V{+}H)$ where $V$ is $\alpha$-strongly convex and $H$ is an $L$-Lipschitz function. {Then \eqref{eq:main_result_equation} holds with 
\begin{align*}
    \Lambda_\infty = \frac{1}{\alpha}\exp\Bigl(\frac{3L^2}{2\alpha} + \frac{6L}{\sqrt{\alpha}}\Bigr)\,.
\end{align*}
If $\mu$ is only $\alpha$-strongly log-concave (i.e., $H=0$), then $\Lambda_\infty = \alpha^{-1}$.}
\end{corollary}
{\begin{remark}
One can view Corollary~\ref{cor:pert} as a strengthening of the classical transport-information inequality. For instance, an application of the HWI inequality (see \cite{gentil2020entropic}) for $\mu$ which is $\alpha$-strongly log-concave yields
\begin{align}\label{eq:hwi_app}
    W_2^2(\mu,\nu) \leq \alpha^{-2}\,{\rm FI}(\nu\mmid\mu)\,,
\end{align}
where $W_2^2(\mu,\nu)$ denotes the squared $2$-Wasserstein distance between $\mu$ and $\nu$ {\cite{otto2000generalization}}. However, it follows from a trivial coupling argument and Corollary~\ref{cor:pert} that
\begin{align*}
    W_2^2(\mu,\nu) \leq \| T_{\rm{KM}}^\mu - T_{\rm{KM}}^\nu\|^2_{L^2(\gamma)} \leq \alpha^{-2}\,{\rm FI}(\nu\mmid\mu)\,.
\end{align*}
Thus, we have strengthened~\eqref{eq:hwi_app} by giving an explicit coupling, and by replacing the $W_2$ metric on the left-hand side with a larger quantity (in fact, the ``linearized'' Wasserstein metric at $\gamma$ between $T_{\rm{KM}}^\mu$ and $T_{\rm{KM}}^\nu$).
\end{remark}}

\begin{proof}[Proof of Corollary~\ref{cor:pert}]
Letting $b = \exp(2{s}) - 1$, carrying out the integration yields
\begin{align*}
    \int_0^{s}\theta_u \dd u = -\frac{1}{2}\log\Bigl(\frac{1+\alpha b}{1+b}\Bigr) + \frac{b L^2}{2\,(1+\alpha b)} + \frac{2 b L}{1+\alpha b}\sqrt{\alpha + b^{-1}}\,.
\end{align*}
Another change of variables with $r = (b+1)^{-1/2} = \exp(-{s})$ yields, for the full integral,
\begin{align*}
 &\int_0^{T} \exp(-{s}) \exp\Bigl(3 \int_0^{{s}} \theta_u \dd u\Bigr)\dd s \\
 &\qquad \qquad =\int_{\exp(-T)}^1 (r^{2} + \alpha\,(1-r^{2}))^{-3/2}\exp\Bigl(\frac{3L^2}{2\,(f(r) + \alpha)}\Bigr)\exp\Bigl(\frac{6L}{\sqrt{f(r) + \alpha}}\Bigr) \dd r\,,
\end{align*}
where $f(r) = r^2/(1-r^2)$. As $f$ is increasing on the interval $(0,1)$ and $f(0) = 0$, we can replace $f(r)  +\alpha \geq \alpha$, and obtain
\begin{align*}
    {\Lambda_T} &\leq \exp\Bigl(\frac{3L^2}{2\alpha} + \frac{6L}{\sqrt{\alpha}}\Bigr) I(T,\alpha) \defeq \exp\Bigl(\frac{3L^2}{2\alpha} + \frac{6L}{\sqrt{\alpha}}\Bigr)  \int_{\exp(-T)}^{1} (r^2 + \alpha\,(1-r^2))^{-3/2} \dd r\,.
\end{align*}
Performing the integration in closed form, it is easy to see that $\lim_{T\to\infty} I(T,\alpha) = \alpha^{-1}$. This concludes the proof.
\end{proof}

\textbf{Example: Gaussian mixtures as tilts.} As an example, we take the case of Gaussian mixtures. Suppose $\mu = \sum_{k=1}^Kw_k\phi(\cdot;m_k,\Sigma)$ where $\phi(\cdot;m_k,\Sigma)$ is the Gaussian density with mean $m_k \in \R^d$ and covariance $\Sigma \succ 0$, and $w_k \geq 0$ are weights (such that $\sum_{k=1}^K w_k = 1$). In this case, it is possible to write down the log-density in the form of our assumptions, with\looseness-1
\begin{align*}
    V(x) = \tfrac{1}{2}\|x\|_{\Sigma^{-1}}^2\,, \quad H(x) = \log\sum_{k=1}^K w_k\exp\bigl(m_k^\top \Sigma^{-1}x - \tfrac12 m_k^\top \Sigma^{-1} m_k \bigr)\,,
\end{align*}
and, moreover, it is easy to verify that
\begin{align*}
    \nabla H(x) = \Sigma^{-1} \sum_{k=1}^K {w}_k(x)\, m_k \defeq \Sigma^{-1} \frac{\sum_{k=1}^K w_k m_k \exp\bigl(m_k^\top \Sigma^{-1}x - \tfrac12m_k^\top \Sigma^{-1}m_k\bigr)}{\sum_{k=1}^K w_k \exp\bigl(m_k^\top \Sigma^{-1}x - \tfrac12m_k^\top \Sigma^{-1}m_k\bigr)}\,.
\end{align*}
If we further assume $\alpha I \preceq \Sigma^{-1} \preceq \beta I$, then via Jensen's inequality,
\begin{align*}
    \|\nabla H(x)\| \leq \|\Sigma^{-1}\|_{\rm op} \max_k \|m_k\| \leq \beta \max_k \|m_k\|\,.
\end{align*}
Thus, for any $\nu$, our stability bound reads
\begin{align*}
\| T_{\rm{KM}}^\mu - T_{\rm{KM}}^\nu\|_{L^2(\gamma)} &\leq \frac{1}{\alpha}\exp\Bigl(O\Bigl(\frac{\beta^2 \max_{k}\|m_k\|^2}{\alpha}\Bigr)\Bigr)\,\sqrt{{\rm FI}(\nu\mmid \mu)}\,.
\end{align*}
\subsubsection{Distributions with asymptotically positive convex profiles}\label{sec:asymp_convex}
As a final example, we turn to a family of distributions  {whose origins date back to \cite{eberle2016reflection,lindvall1986coupling}}. To define said family, we require the following definitions. For a function $f:\R^d \to \R$, we define the integrated convexity profile of $f$, denoted $\kappa_f :\R_+ \to \R\cup\{-\infty\}$, to be 
\begin{align*}
    \kappa_f(r) \defeq \inf\Bigl\{ \frac{\langle \nabla {f}(x) - \nabla {f}(y), x-y \rangle}{\|x-y\|^2} \ : \ \|x-y\| = r\Bigr\}\,.
\end{align*}
We {motivate} this definition with the following example.

\textbf{Example: Strongly convex potential outside a ball.}
Taking $\mu\propto\exp(-V)$, suppose that there exists $\alpha_V > 0 $ and $R_V,L_V \geq 0$ such that
\begin{align}\label{eq:asymp_1}
    \kappa_V(r) \geq 
    \begin{cases}
    \alpha_V & {\text{ for } } r > R_V\,, \\
    \alpha_V - L_V & {\text{ for }  }r \leq R_V\,.
    \end{cases}
\end{align}
If $R_V = 0$, then $\kappa_V(r) \geq \alpha_V > 0$ is equivalent to strong convexity over all of $\R^d$. Otherwise for $R_V > 0$, the potential has integrated convexity profile which might be negative inside $B(0,R_V)$, while remaining strongly convex outside the ball.

\medskip{}

The following proposition gives an alternative characterization of \eqref{eq:asymp_1}. 
\begin{proposition}[{\citep[Proposition 5.1]{conforti2024weak}}]\label{prop:conforti_prop}
Suppose $V$ satisfies \eqref{eq:asymp_1} with constants $\alpha_V > 0$ and $L_V, R_V \geq 0$. Then it satisfies, for all $r > 0$,
\begin{align*}
    \kappa_V(r) \geq \alpha_V - r^{-1}\hat g_{\hat L}(r)\,, 
\end{align*}
where $\hat g_L(r) \deq 2\sqrt{L}\tanh(r\sqrt{L})$, and $\hat L$ is given by
\begin{align*}
    \hat L \defeq \begin{cases}
        \inf\{ L\ge 0 : R_V^{-1}\hat g_L(R_V) \geq L_V\}\,, & R_V > 0\,, \\
        0\,, & R_V = 0\,.
    \end{cases}
\end{align*}
\end{proposition}
Following Proposition~\ref{prop:conforti_prop}, it is worth considering some asymptotic scenarios for determining $\hat L$. For instance, if $R_V^2L_V \ll 1$, then one can verify that $\hat L \approx L_V/2$. On the other hand, if $R_V^2L_V \gg 1$, then $\hat{L} \approx L_V^2R_V^2/4$. 

To generalize this characterization, we can consider functions $g \in \cG \subset C^2((0,\infty),\R_+)$ if they satisfy the following properties:
\begin{enumerate}
    \item $r\mapsto r^{1/2}g(r^{1/2})$ is non-decreasing and concave,
    \item $\lim_{r\downarrow 0} rg(r) = 0$,
    \item $g$ itself is bounded such that $g' \geq 0$ and $2 g'' + g g' \leq 0$,
    \item the right-derivative of ${g}$ at the origin exists, denoted ${g'(0)}$.
\end{enumerate}
This function class leads to our final corollary {of this section}.

\begin{corollary}[Asymptotically positive convex profile]\label{cor:asymp_conv} 
Let ${\alpha} > -1$ and $\hat{g} \in \cG$, and $\mu \propto \gamma \exp(-h)$ where the integrated convexity profile of $h$ satisfies $\kappa_h(r) \geq \alpha - r^{-1}\hat{g}(r)$. Then,\looseness-1
{\begin{align*}
    \Lambda_\infty = \frac{1}{1+\alpha}\exp\Bigl(\frac{3\hat g'(0)}{2\,(1+\alpha)}\Bigr)\,.
\end{align*}}
\end{corollary}
\begin{proof}
By \cite[Lemma 5.9]{conforti2025projected}, in this case it holds that
\begin{align*}
    \theta_u = -\frac{\exp(-2u)}{1 + (1-\exp(-2u))\,\alpha}\Bigl(\alpha - \frac{\hat{g}'(0)}{1 + (1-\exp(-2u))\,\alpha}\Bigr)\,.
\end{align*}
Mimicking the computations from before, we then compute 
\begin{align*}
 {\Lambda_T} &\leq \exp\Bigl(-\frac{3\hat g'(0)}{2\alpha}\bigl((1+\alpha)^{-1} - 1)\bigr)\Bigr)\int_{\exp(-T)}^1 (1 + \alpha -\alpha r^{2})^{-3/2} \dd r\,,
\end{align*}
where we obtain the result by taking the $T\to\infty$ limit.
\end{proof}

Returning to our example of log-densities which are strongly convex outside a ball, we can instantiate Corollary~\ref{cor:asymp_conv} with the heuristic bounds on $\hat L$ to obtain the bounds
\begin{align*}
    \| T_{\rm{KM}}^\mu - T_{\rm{KM}}^\nu\|_{L^2(\gamma)} \leq 
    {\alpha_V}^{-1} \exp\bigl(O\bigl((L_V/{\alpha_V})\,(1 \vee L_V R_V^2)\bigr)\bigr)\,\sqrt{{\rm FI}(\nu \mmid \mu)}\,.
\end{align*}
\section{Extension to stronger metrics}
In this section, we show how our proof above can be modified to provide stability bounds under stronger metrics. Writing $\infFI(\nu\mmid\mu) \defeq {\|\nabla\log(\nu/\mu)\|_{L^\infty(\nu)}^2}$, our goal now is to establish bounds of the form
\begin{align}\label{eq:esssup_result}
    \|T^\mu_{\rm KM} - T^\nu_{\rm KM}\|_{L^\infty(\gamma)} \lesssim \sqrt{\infFI(\nu\mmid\mu)}\,,
\end{align}
where again the hidden constants will be made explicit.

To start, we follow the start of the proof of Theorem~\ref{thm:l2_result} assuming \eqref{eq:assump}. It is easy to see that $\gamma$-almost surely, it holds that
\begin{align*}
    \partial_t \|X_t - Y_t\|^2 \leq 2\theta_{T - t}\,\|X_t - Y_t\|^2 + 2\,\|X_t - Y_t\|\infFI(q_{T-t}^\nu\mmid q_{T-t}^\mu)^{1/2}\,.
\end{align*}
Instead of relying on Lemma~\ref{lem:andre} (for which an analogue for $\infFI$ is out of scope), suppose instead that for some constant $d_{T-t} \geq 0$
\begin{align}\label{eq:infFI_ineq}
    \sqrt{\infFI(q_{T-t}^\nu\mmid q_{T-t}^\mu)} \leq d_{T-t}\sqrt{\infFI(\nu\mmid \mu)}\,.
\end{align}
{We see that the rest of the argument is carried out as before, replacing $\Lambda_T$ with}
\begin{align}\label{eq:esssup_proof}
    {\eta_T \defeq} \int_0^T d_{s}\exp\Bigl(\int_0^{{s}}\theta_u \dd u \Bigr)\dd s\,.
\end{align}
It remains to show that \eqref{eq:infFI_ineq} holds, after which taking limits in~\eqref{eq:esssup_proof} would yield our desired result.

To this end, {for a measure $\mu$,} $y \in \R^d$ and $t \geq 0$, we write
\begin{align*}
    \mu_{y,t}(x) \propto q_t(y\mid x)\,\mu(x)\,,\qquad q_t(\cdot\mid x) = \cN(e^{-t}x, (1-e^{-2t})I)\,.
\end{align*}
Finally, we recall that $\mu$ satisfies a log-Sobolev inequality with constant $\lambda>0$ if
\begin{align*}
    {\rm KL}(\nu\mmid\mu) \defeq \int {\log \Bigl(\frac{\nu}{\mu}\Bigr)\dd\nu} \leq \frac{\lambda}{2}\,{\rm FI}(\nu\mmid\mu)\,.
\end{align*}
We will now prove the following.

\begin{lemma}\label{lem:stronger_lemma}
Suppose that $\mu$ is such that for all $y \in \R^d$ and $t \geq 0$, $\mu_{y,t}$ satisfies a log-Sobolev inequality with constant $\lambda_t$. Then, for all $s \geq 0$,
\begin{align*}
    \sqrt{\infFI(q^\nu_s\mmid q^\mu_s)} \leq \frac{e^{s}\lambda_s}{e^{2s}-1} \sqrt{\infFI(\nu \mmid \mu)}\,.
\end{align*}
\end{lemma}
\begin{proof}
Recalling that $\rho Q_t(y) = \int q_t(y\mid x)\dd\rho(x)$, it is easy to see that 
\begin{align*}
    \nabla \log \frac{\nu Q_s}{\mu Q_{s}}(y) = \frac{e^{-s}}{1-e^{-2s}} \bigl( \E_{\nu_{y,s}}[X] - \E_{\mu_{y,s}}[X]\bigr)\,,
\end{align*}
and thus in norm
\begin{align*}
    \Bigl\|\nabla \log \frac{\nu Q_s}{\mu Q_s}\Bigr\|_{L^\infty(\nu Q_s)} = \frac{e^{-s}}{1-e^{-2s}} \sup_{y \in \R^d}{\Bigl\lVert\int x \dd( \nu_{y,s} - \mu_{y,s})(x)\Bigr\rVert} \,.
\end{align*}
We can further bound the right-hand side as
\begin{align*}
    \Bigl\|\int x \dd( \nu_{y,s} - \mu_{y,s})(x)\Bigr\| \leq  W_2(\nu_{y,s},\mu_{y,s})\,.
\end{align*}
If $\mu_{y,s}$ satisfies a log-Sobolev inequality with constant $\lambda_s$, it also satisfies the following transport-information inequality:
\begin{align*}
    W_2^2(\nu_{y,s},\mu_{y,s}) \leq \lambda_s^2\, {\rm FI}(\nu_{y,s}\mmid\mu_{y,s})\,.
\end{align*}
(This is a generalization of \eqref{eq:hwi_app}; see \cite{gentil2020entropic}.) By the definitions of $\mu_{y,s}$ and $\nu_{y,s}$, 
\begin{align*}
    \nabla \log \frac{\nu_{y,s}}{\mu_{y,s}} = \nabla \log \frac{\nu}{\mu}\,.
\end{align*}
As $\sup_{y\in\R^d} {\rm FI}(\nu_{y,s}\mmid \mu_{y,s}) \leq \infFI(\nu\mmid\mu)$, our proof is complete.
\end{proof}

We now instantiate Lemma~\ref{lem:stronger_lemma} for our existing examples to obtain contraction estimates in $\infFI$.
\begin{proposition}\label{prop:stronger_prop}
For any $s \geq 0$, write $u(s) = e^{2s}-1$. If $\mu$ satisfies the conditions of
\begin{enumerate}
    \item[(a)] Corollary~\ref{cor:pert}, then Lemma~\ref{lem:stronger_lemma} holds with 
    \begin{align}\label{eq:lsi_pert}
            \lambda_{s} = (\alpha + 1/u(s))^{-1}\exp\Bigl(\frac{L^2}{\alpha+1/u(s)} + \frac{4L}{(\alpha+1/u(s))^{1/2}}\Bigr)\,;
    \end{align}
    \item[(b)] Corollary~\ref{cor:asymp_conv}, then Lemma~\ref{lem:stronger_lemma} holds with 
\begin{align}\label{eq:lsi_asymp}
    \lambda_s = (1 + \alpha + 1/u(s))^{-1}\exp\Bigl(\frac{\hat g'(0)}{1 + \alpha + 1/u(s)}\Bigr)\,.
\end{align}
\end{enumerate}
\end{proposition}
\begin{proof}
Suppose $\mu$ satisfies the assumptions of Corollary~\ref{cor:pert}, i.e., $\mu\propto\exp(-{V+H})$ where $V$ is $\alpha$-strongly convex and $H$ is $L$-Lipschitz. Then it is easy to see that for any $y \in \R^d$ and $s \geq 0$ that $\mu_{y,s}$ 
is strongly log-concave with parameter $\alpha + 1/u(s)$ and the log-perturbation $H$ remains unchanged. Thus by \cite[Theorem 1.4]{BriPed25HeatFlow}, $\mu_{y,s}$ satisfies the log-Sobolev inequality with parameter given by \eqref{eq:lsi_pert}. 

The argument for the second case is identical---the tilted measure is more strongly log-concave everywhere. In this case, the log-Sobolev constant is given by \eqref{eq:lsi_asymp}; see \cite[Theorem 5.7]{conforti2025projected}.\qedhere
\end{proof}

Combined with the computations surrounding \eqref{eq:infFI_ineq} and \eqref{eq:esssup_proof}, we now arrive at the following result.

\begin{theorem}\label{thm:inffi}
    Suppose $\mu$ is such that \eqref{eq:assump} holds for some constants $(\theta_s)_{s \geq 0}$ and that $\mu_{y,t}$ satisfies a log-Sobolev inequality with constant ${\lambda_t}$ for any $y \in \R^d$ and $t \geq 0$. Then $\gamma$-a.s.,
    \begin{align*}
        \|T^\mu_{\rm KM} - T^\nu_{\rm KM}\|_{L^\infty(\gamma)} {\leq \eta_\infty\sqrt{\infFI(\nu\mmid\mu)} \defeq \Bigl[\int_0^\infty \frac{e^T \lambda_T}{e^{2T}-1} \exp\Bigl(\int_0^T \theta_u\dd u\Bigr)\dd T\Bigr]} \sqrt{\infFI(\nu\mmid\mu)}\,.
    \end{align*}
\end{theorem}
{
\begin{remark}
To the best of our knowledge, Theorem~\ref{thm:inffi} establishes a new transport-information inequality of the form 
\begin{align*}
    W_\infty(\mu,\nu) \leq \|T^\mu_{\rm KM} - T_{\rm KM}^\nu\|_{L^\infty(\gamma)}  \leq \eta_\infty\sqrt{\infFI(\nu\mmid\mu)}\,,
\end{align*}
where we recall that $W_\infty(\mu,\nu) \defeq \inf_{\pi \in \Pi(\mu,\nu)} \esssup_{(X,Y)\sim\pi} \|X-Y\|$,  
where $\Pi(\mu,\nu)$ is the set of joint measures with first- and second-marginal given by $\mu$ and $\nu$ respectively.    
\end{remark}
We will now revisit our previous examples (log-Lipschitz perturbations and strongly log-concave outside a ball) in the context of Theorem~\ref{thm:inffi}, where we obtain the same constants as in Section~\ref{sec:main_results}.
}
\begin{corollary}\label{cor:pert_inf}
Suppose $\mu\propto\exp(-{V +H})$ where $V$ is $\alpha$-strongly convex and $H$ is $L$-Lipschitz. Then Theorem~\ref{thm:inffi} holds with constant
\begin{align*}
    \eta_\infty = \alpha^{-1}\exp\Bigl(\frac{3L^2}{2\alpha} + \frac{6L}{\sqrt{\alpha}}\Bigr)\,.
\end{align*}
\end{corollary}
\begin{proof}
By {Lemma~\ref{lem:stronger_lemma} and Proposition~\ref{prop:stronger_prop}}, we can compute $d_s$ and use the bound $\alpha + 1/u(s) \geq \alpha$ for all $s \geq 0$ to obtain
\begin{align*}
    d_s \leq \frac{e^s}{e^{2s}-1}\,\frac{1}{\alpha + (e^{2s}-1)^{-1}}\exp\Bigl(\frac{L^2}{\alpha} + \frac{4L}{\sqrt{\alpha}}\Bigr) = \frac{e^s}{\alpha\,(e^{2s} - 1) + 1}\exp\Bigl(\frac{L^2}{\alpha} + \frac{4L}{\sqrt{\alpha}}\Bigr) \,,
\end{align*}
which then leads to
\begin{align*}
    \eta_T \leq \exp\Bigl(\frac{L^2}{\alpha} + \frac{4L}{\sqrt{\alpha}}\Bigr) \int_0^T \frac{e^{{s}}}{\alpha\,(e^{2{s}}-1) + 1}\exp\Bigl(\int_0^{{s}}\theta_u \dd u\Bigr)\dd s\,. 
\end{align*}
We already computed the integral inside the exponential (in the proof of Corollary~\ref{cor:pert}). Dropping the same terms, we obtain the following upper bound
\begin{align*}
    \eta_T \leq \exp\Bigl(\frac{3L^2}{2\alpha} + \frac{6L}{\sqrt{\alpha}}\Bigr) \int_0^T \frac{e^{2{s}}}{(\alpha e^{2{s}} + (1-\alpha))^{3/2}} \dd s\,. 
\end{align*}
One can obtain our final result by evaluating the integral by elementary means, and taking the limit as $T\to\infty$.
\end{proof}

\begin{corollary}
Suppose $\mu\propto \gamma\exp(-h)$ where the potential satisfies the conditions in Corollary~\ref{cor:asymp_conv}. Then Theorem~\ref{thm:inffi} holds with constant
\begin{align*}
    \eta_\infty = (1+\alpha)^{-1}\exp\Bigl(\frac{3\hat g'(0)}{2(1+\alpha)}\Bigr)\,.
\end{align*}
\end{corollary}
\begin{proof}
The computations follow verbatim the arguments from Corollary~\ref{cor:pert_inf}, using the second part of {Proposition~\ref{prop:stronger_prop}}, and are thus omitted.
\end{proof}

\bibliography{ref}

\end{document}